\documentclass[12pt, letterpaper]{amsart}

\title{Scattering resonances on truncated cones}
\author{Dean Baskin}
\address{Department of Mathematics, Texas A\&M University}
\email{dbaskin@math.tamu.edu}
\author{Mengxuan Yang}
\address{Department of Mathematics, Northwestern University}
\email{mxyang@math.northwestern.edu}

\usepackage{color}

\usepackage[margin=1in]{geometry}
\usepackage{mathtools}
\DeclarePairedDelimiter{\floor}{\lfloor}{\rfloor}

\newcommand{\lap}{\Delta}
\newcommand{\pd}[1][]{\partial_{#1}}
\newcommand{\complexes}{\mathbb{C}}
\newcommand{\reals}{\mathbb{R}}
\newcommand{\sphere}{\mathbb{S}}
\DeclareMathOperator{\vol}{Vol}
\DeclareMathOperator{\supp}{supp}
\newcommand{\rt}{\widetilde{r}}
\newcommand{\yt}{\widetilde{y}}
\newcommand{\norm}[2][]{\| #2\|_{#1}}

\newcommand{\bo}{O}
\newcommand{\lo}{o}

\DeclareMathOperator{\Ai}{Ai}

\newtheorem{theorem}{Theorem}

\newtheorem{lemma}[theorem]{Lemma}

\numberwithin{theorem}{section}

\begin{document}

\begin{abstract}
We consider the problem of finding the resonances of the Laplacian on
truncated Riemannian cones.  In a similar fashion to Cheeger--Taylor,
we construct the resolvent and scattering matrix for the Laplacian on
cones and truncated cones.  Following Stefanov, we show that the
resonances on the truncated cone are distributed asymptotically as
$Ar^{n} + o(r^{n})$, where $A$ is an explicit coefficient.  We also
conclude that the Laplacian on a non-truncated cone has no resonances.
\end{abstract}

\maketitle

\section{Introduction}
\label{sec:introduction}

In this note, we consider the  resonances on truncated
Riemannian cones and establish a Weyl-type formula for their
distribution.  To fix notation, we let $(Y,h)$ be a compact
$(n-1)$-dimensional Riemannian manifold (with or without boundary) and
let $C(Y)$ denote the cone over $Y$.  In other words, $C(Y)$ is
diffeomorphic to the product $(0,\infty)_{r}\times Y$ and is equipped
with the incomplete Riemannian metric $g = dr^{2} + r^{2}h$.  We refer
the reader to the foundational work of Cheeger--Taylor~\cite{CT1, CT2}
for more details on the geometric set-up.  We also introduce the
\emph{truncated} Riemannian cone $C_{a}(Y)$ formed by introducing a
boundary at $r=a$, i.e., $C_{a}(Y)$ is diffeomorphic to
$[a, \infty)_{r} \times Y$ and equipped with the same metric.

The (negative-definite) Laplacian on $C(Y)$ (or $C_{a}(Y)$ with a
choice of boundary conditions) has the
form
\begin{equation*}
  \pd[r]^{2} + \frac{n-1}{r}\pd[r] + \frac{1}{r^{2}}\lap_{h},
\end{equation*}
where $\lap_{h}$ denotes the Laplacian of $(Y,h)$.  Its resolvent $R(\lambda)$ is given by 
\begin{equation*}
R(\lambda) = (\lap + \lambda^2)^{-1}.
\end{equation*}
We consider the \emph{cutoff resolvent} $\chi R(\lambda)\chi$, where
$\chi$ is a (fixed) smooth compactly supported function on $C(Y)$ (or
$C_a(Y)$).  One consequence of the resolvent
formula of Theorem~\ref{thm:resolvent} is that the cutoff resolvent extends
meromorphically to the logarithmic cover of $\complexes\setminus \{ 0 \}$.

More precisely, we identify elements $\lambda$ of the logarithmic
cover of $\complexes \setminus \{0 \}$ by a magnitude $|\lambda|$ and
a phase $\arg \lambda \in \reals$.  We identify the ``physical
half-plane'' as those $\lambda$ with $\arg \lambda \in (0, \pi)$.
These $\lambda$ correspond to the resolvent set
$\complexes \setminus [0, \infty)$ via the map
$\lambda \mapsto |\lambda|^{2}e^{2i\arg\lambda}$.  The cutoff
resolvent then extends to be meromorphic as a function of $\lambda$ on
this logarithmic cover.

The poles of the cutoff resolvent consist of possibly finitely many
$L^2$-eigenvalues lying in the upper half-plane (which do not appear
with Dirichlet boundary conditions) and poles lying on other sheets of
the cover.  The latter poles are called the \emph{resonances} of
$\lap$.

The main theorem of this paper counts the most physically relevant
resonances for the truncated cone.  In particular, we count those
resonances $\lambda$ nearest to the physical half-plane, i.e., those
with $\arg \lambda \in (-\frac{\pi}{2}, 0)$ and
$\arg \lambda \in (\pi, \frac{3\pi}{2})$.  The resonances on other
``sheets'' of the cover remain more mysterious and are related to
the zeros of Hankel functions near the real axis.  We consider the
resonance counting function on these sheets, defined by
\begin{equation*}
  N(r) = \# \left\{ \lambda : \lambda \text{ is a resonance and
    }|\lambda| \leq r\right\}.
\end{equation*}
The following theorem provides an asymptotic formula for $N(r)$.
\begin{theorem}
\label{thm:main-thm}
Suppose either that the set of periodic geodesics of $(Y,h)$ has
Liouville measure zero or that $Y = \sphere^{n-1}$ equipped with a
constant rescaling of the standard metric.  
Consider the truncated cone $C_1(Y)$ equipped with the Dirichlet
Laplacian and let $N(r)$ denote its resonance counting function on the
neighboring sheets as above.  We then have, as $r\to \infty$,
\begin{equation*}
N(r) = A_{n}\vol (Y,h) r^n + \lo (r^n),
\end{equation*}
where $A_{n}$ is an explicit constant (defined below in
equation~\eqref{eq:defn-ofAn}) and $\vol(Y,h)$ denotes the volume of
the Riemannian manifold $(Y,h)$.
\end{theorem}

The constant $A_{n} \vol (Y,h)$ in Theorem~\ref{thm:main-thm} is the
same constant as computed by Stefanov~\cite{Stefanov} for the
resonance counting function on the domain exterior to a ball in
$\reals^{n}$.  When $Y = \mathbb{S}^{n-1}$ is equipped with its
standard metric, the truncated cone $C_{1}(Y)$ can can be thought of
as the exterior of the unit ball in Euclidean space.
Theorem~\ref{thm:main-thm} recovers Stefanov's result.  (When $Y =
\mathbb{S}^{n-1}$, $n$ odd, is equipped with its standard metric, the
cutoff resolvent in fact continues to the complex plane; this can be
seen in the resolvent formulae below.)

We also state the following theorem, which is known to the community
but does not seem to be in the literature.
\begin{theorem}
  \label{thm:full-cone}
  If $(Y,h)$ is a compact Riemannian manifold (with or without
  boundary) then the cone $C(Y)$ has no resonances.
\end{theorem}

In fact, Theorem~\ref{thm:resolvent} below shows that $\lambda$ is a
resonance of the truncated cone $C_{1}(Y)$ if and only if $\lambda /
a$ is a resonance of the truncated cone $C_{a}(Y)$.  Sending $a$ to $0$
then pushes all resonances out to infinity and provides evidence for
Theorem~\ref{thm:full-cone}.  

The proof of Theorem~\ref{thm:main-thm} has two main steps.  We first
separate variables and obtain an explicit resolvent formula
in Theorem~\ref{thm:resolvent} to characterize the resonances as zeros
of a Hankel function.  In Section~\ref{sec:reson-count-trunc} we
consider the asymptotic distribution of the zeros of each Hankel
function appearing in the resolvent formula.  The hypothesis on the
link $(Y,h)$ is used to control the error terms when synthesizing the result.
Theorem~\ref{thm:full-cone} is an immediate corollary of the resolvent
formula in Theorem~\ref{thm:resolvent}.

The proof of Theorem~\ref{thm:main-thm} follows an argument of
Stefanov~\cite{Stefanov} very closely.  Stefanov established a
Weyl-type law for the distribution of resonances for the exterior of a
ball in odd-dimensional Euclidean space.  The main contribution of
this paper is the observation that, after some natural modifications,
the core of Stefanov's argument applies to the setting of cones.
Borthwick~\cite{borthwick:2010, borthwick:2012} and
Borthwick--Philipp~\cite{Borthwick-Philipp} showed that a similar
approach works in the asymptotically hyperbolic setting.  

We further remark that we have specialized to the Dirichlet Laplacian
in Theorem~\ref{thm:main-thm} only for simplicity.  For Neumann or
Robin boundary conditions, the resolvent formula of
Theorem~\ref{thm:resolvent} has an analogous expression.  The
resonance counting problem then involves counting zeros of
$H^{(2)'}_{\nu} + C\nu H^{(2)}_{\nu}$, which can be handled with
similar arguments.

\section{Resolvent construction}

In this section we write down an explicit formula (via separation of
variables) for the resolvent and then show that the cut-off resolvent
has a meromorphic continuation to the logarithmic cover $\Lambda$ of
the complex plane.  The construction is essentially contained in the
work of Cheeger--Taylor~\cite{CT1,CT2}, but the resolvent is not
explicitly written there.

Suppose that $\phi_j$ form an orthonormal family of eigenfunctions for
$-\lap_h$ with corresponding eigenvalues $\mu_j^2$.  We decompose
$L^{2}(C(Y))$ into a direct sum in terms of the eigenspaces of
$-\lap_{h}$, i.e.,
\begin{equation*}
  L^{2}(C_{a}(Y); \complexes) = \bigoplus_{j=0}^{\infty}L^{2}((a,
  \infty); E_{j}), \quad f(r,y) = \sum_{j=0}^{\infty}f_{j}(r) \phi_{j}(y),
\end{equation*}
where the first space is defined with respect to the volume form
induced by the metric and the latter spaces can be identified (via the
identification $f(r)\phi_{j}(y) \mapsto f(r)$) with the space
$L^{2}((a,\infty); \complexes)$ equipped with the volume form
$r^{n-1}\,dr$.

For $\arg \lambda \in (0,\pi)$, the resolvent $R(\lambda)$ splits as a
direct sum of operators $R_{j}(\lambda)$ acting on
$L^{2}((a,\infty), E_{j})$, with measure $r^{n-1}\,dr$.
\begin{equation*}
R(\lambda)\left( \sum_{j=1}^\infty f_j(r) \phi_j(y)\right) = \bigoplus_{j=1}^\infty \left( R_j(\lambda)f_j \right)\phi_j (y).
\end{equation*}

In this section, we prove the following explicit formula for the
$j$-th piece of the resolvent.  For the cone $C(Y)$ (i.e., for $a=0$),
we use the Friedrichs extension of the Laplacian to guarantee
self-adjointness (though in high enough dimension the Laplacian is
essentially self-adjoint):

\begin{theorem}
\label{thm:resolvent}
The piece of the resolvent corresponding to the $j$-th eigenvalue has
the following explicit expression on the truncated cone $C_{a}(Y)$ or
the cone $C(Y)$ ($a=0$):
\begin{equation*}
  (R_{j}(\lambda) f)(r) = \int_{a}^{\infty}K_{a,j}(r,\tilde{r}) f(\tilde{r}) \tilde{r}^{n-1}\,d\tilde{r}
\end{equation*}
where $K_{a,j}(r,\tilde{r})$ is given by
\begin{equation*}
  K_{a,j}(r,\tilde{r}) = \frac{\pi}{2i}\left(
    \tilde{r}r\right)^{-\frac{n-2}{2}}
  \begin{cases}
    H^{(1)}_{\nu_{j}}(\lambda \tilde{r}) J_{\nu_{j}}(\lambda r) -
    \frac{J_{\nu_{j}}(\lambda a)}{H^{(1)}_{\nu_{j}}(\lambda
      a)}H^{(1)}_{\nu_{j}}(\lambda \tilde{r})
    H^{(1)}_{\nu_{j}}(\lambda r) & r < \tilde{r} \\
    J_{\nu_{j}}(\lambda \tilde{r}) H^{(1)}_{\nu_{j}} (\lambda r) -
    \frac{J_{\nu_{j}}(\lambda a)}{H^{(1)}_{\nu_{j}}(\lambda
      a)}H^{(1)}_{\nu_{j}}(\lambda \tilde{r})
    H^{(1)}_{\nu_{j}}(\lambda r) & r > \tilde{r}
  \end{cases}
\end{equation*}
Here $J_{\nu}$ are the standard Bessel functions of the first kind and
$H^{(1)}_{\nu}$ are the Hankel functions of the first kind.  The
second term in both expressions should be interpreted as $0$ when
$a=0$.  
\end{theorem}

\begin{proof}
  After separating variables, we may assume that
  $f = f_j(r)\phi_j(y)$.  We construct the resolvent for
  $\Im \lambda > 0$ and then meromorphically continue the expression.

  Writing $u = u_j (r) \phi_j(y)$, the equation
  $(  \lap + \lambda^2) u = f$ induces the following differential
  equation for $u_j$:
  \begin{equation}
    \label{eq:inhomog}
    \pd[r]^2 u_j+ \frac{n-1}{r}\pd[r]u_j - \frac{\mu_j^2}{r^2}u_j + \lambda^2 u_j = f_j.
  \end{equation}
  We solve this equation by showing it is equivalent to a Bessel
  equation.  

  Changing variables to $\rho = \lambda r$ and writing
  $\tilde{u}(\rho) = u(\rho / \lambda)$ yields
  \begin{equation*}
    \pd[\rho]^2\tilde{u} + \frac{n-1}{\rho}\pd[\rho]\tilde{u} +
    \left( 1 - \frac{\mu_j^2}{\rho^2}\right) \tilde{u} = \frac{1}{\lambda^{2}}\tilde{f}(\rho).
  \end{equation*}
  Writing $v = \rho ^{(n-2)/2}\tilde{u}$, we obtain a Bessel
  equation for $v$:
  \begin{equation}
    \label{eq:Bessel}
    v '' + \frac{1}{\rho}v' + \left( 1 - \frac{\nu_j^2}{\rho^2}\right) v =
    g(\rho), 
  \end{equation}
  where $\nu_j^2 = \mu_j^2 + \left( \frac{n-2}{2}\right)^2$ and
  $g(\rho ) = \frac{\rho^{(n-2)/2}}{\lambda^2}\tilde{f}(\rho)$.  

  We now proceed by the standard ODE technique of variation of
  parameters.  One basis for the space of solutions of the homogeneous
  version of this Bessel equation is
  $\{ J_{\nu_{j}}(\rho), H^{(1)}_{\nu_{j}}(\rho)\}$, where $J_{\nu}$
  is the Bessel function of the first kind and $H^{(1)}_{\nu}$ is the
  Hankel function of the first kind.  We thus may use the following
  basis for the space of solutions of the homogeneous
  equation:
  \begin{equation}
    \label{eq:basis}
    w_{1}(r) =  r^{-(n-2)/2}J_{\nu_{j}}(\lambda r), \quad w_{2}(r) = r^{-(n-2)/2}H^{(1)}_{\nu_{j}}(\lambda r)
  \end{equation}
  
  For $\Im \lambda > 0$, $R_{j}(\lambda) f_{j}$ must lie in
  $L^{2}((a,\infty), r^{n-1}\,dr)$.  If $f_{j}$ is compactly
  supported, this means that $u_{j} = R_{j}(\lambda)f_{j}$ must be a
  multiple of $r^{-(n-2)/2}H^{(1)}_{\nu_{j}}(\lambda r)$ near
  infinity.  When $a >0$, $u_{j}$ must satisfy the boundary condition
  at $r=a$.  When $a=0$, the choice of the Friedrichs extension
  requires that both $u_{j}$ and $u_{j}'$ lie in the the weighted
  $L^{2}$ space near $0$ and so $u_{j}$ must be a multiple of
  $r^{-(n-2)/2}J_{\nu_{j}}(\lambda r)$ near $r=0$ as any nonzero
  multiple of $w_{2}$ will not have this property.
  
  We may thus write
  \begin{equation*}
    u_{j}(r) = \left( \int_{r}^{\infty}
      \frac{w_{2}(\tilde{r})f_{j}(\tilde{r})}{W(w_{1}, w_{2})(\tilde{r})}\, d\tilde{r} \right) w_{1}(r) +
    \left( C + \int_{a}^{r} \frac{w_{1}(\tilde{r})f_{j}(\tilde{r})}{W(w_{1},
        w_{2})(\tilde{r})}\,d\tilde{r}\right) w_{2}(r),
  \end{equation*}
  where $C$ is a yet-to-be-determined constant, the functions $w_{1}$
  and $w_{2}$ are as in equation~\eqref{eq:basis}, and $W(w_{1}, w_{2})$ is their
  Wronskian.  The Wronskian $W$ can be easily computed in terms of the
  Wronskian of the Bessel and Hankel functions and seen to be
  \begin{equation*}
    W(w_{1}, w_{2})(r) = r^{-(n-1)}\cdot \frac{2i}{\pi}.
  \end{equation*}
  
  We now turn our attention to the boundary condition.  For $a = 0$,
  the requirement that the solution and its derivative live in $L^{2}$
  forces $C = 0$, yielding the result.  For $a \neq 0$, we require
  that $u_{j}(a) = 0$, i.e.,
  \begin{equation*}
    \left( \frac{\pi}{2i} \int_{a}^{\infty} H^{(1)}_{\nu_{j}}(\lambda
      \tilde{r}) \tilde{r}^{\frac{n}{2}}f(\tilde{r}) \,d\tilde{r}\right) a^{-(n-2)/2} J_{\nu_{j}} (\lambda
    a) + C a^{-(n-2)/2}H^{(1)}_{\nu_{j}}(\lambda a) = 0, 
  \end{equation*}
  and so we must have
  \begin{equation*}
    C = - \frac{\pi}{2i} \frac{J_{\nu_{j}}(\lambda
      a)}{H^{(1)}_{\nu_{j}}(\lambda a)} \int_{a}^{\infty}
    H^{(1)}_{\nu_{j}}(\lambda \tilde{r}) \tilde{r}^{\frac{n}{2}}f(x) \, dx,
  \end{equation*}
  finishing the proof.
\end{proof}

We now claim that $\chi R(\lambda)\chi$ has a meromorphic continuation:
\begin{lemma}
Given a fixed $\chi \in C^\infty_c(\reals_+ \times Y)$, $\chi R(\lambda) \chi$ meromorphically continues from 
\begin{equation*}
\{ \lambda \in \mathbb{C} : \Im \lambda > 0 \}
\end{equation*}
to the logarithmic cover $\Lambda$ of the complex plane.
\end{lemma}

\begin{proof}
  We first prove the statement for the full cone; the statement for
  the truncated cone will follow by an appeal to the analytic Fredholm
  theorem.

  Fix $\chi \in C^{\infty}_{c}((0,\infty))$ and regard $\chi (r)$ as a
  compactly supported smooth function on $C(Y)$.  We let $R(\lambda)$
  denote the resolvent on the non-truncated cone (i.e., $a=0$) and
  $K(\lambda; r, y, \rt, \yt)$ denote its integral kernel.  In order
  to show that $\chi R(\lambda)\chi$ meromorphically continues, it
  suffices to show that for any $f, g \in L^{2}(C(Y))$, the function
  \begin{equation*}
    \lambda \mapsto \langle \chi R(\lambda) \chi f, g\rangle
  \end{equation*}
  meromorphically continues to $\Lambda$.  

  Fix two such functions $f, g\in L^{2}(C(Y))$ and let $f_{j}(r)$ and
  $g_{j}(r)$ denote their coefficients in the expansion in terms of
  eigenfunctions of $\lap_{h}$, i.e.,
  \begin{equation*}
    f(r,y) = \sum_{j=0}^{\infty}f_{j}(r) \phi_{j}(y).
  \end{equation*}
  We observe that because $f$ and $g$ are square-integrable, the sum
  and the integral commute, i.e., 
  \begin{equation*}
    \norm[L^{2}(C(Y))]{f}^{2} = \int_{0}^{\infty}
    \sum_{j=0}^{\infty}|f_{j}(r)|^{2} r^{n-1}\,dr =
    \sum_{j=0}^{\infty} \int_{0}^{\infty}|f_{j}(r)|^{2} r^{n-1}\,dr.
  \end{equation*}

  From Theorem~\ref{thm:resolvent}, we may write
  \begin{align}
    \label{eq:pair-expansion}
    \langle \chi R(\lambda) \chi f, g \rangle &=
                                                \sum_{j=0}^{\infty} \left(
                                                \int_{0}^{\infty}\int_{0}^{r}
                                                (\rt r)^{-\frac{n-2}{2}}
                                                \chi(r)\chi(\rt)
                                                f_{j}(\rt) g_{j}(r)
                                                J_{\nu_{j}}(\lambda
                                                \rt)
                                                H_{\nu_{j}}^{(1)}(\lambda
                                                r) \rt^{n-1} r^{n-1}
                                                \,d\rt\,dr  \right. \notag\\
    &\quad \left. + \int_{0}^{\infty}\int_{r}^{\infty} (\rt
      r)^{-\frac{n-2}{2}} \chi(r) \chi(\rt) f_{j}(\rt) g_{j}(r)
      J_{\nu_{j}}(\lambda r) H_{\nu_{j}}^{(1)}(\lambda \rt) \rt^{n-1}
      r^{n-1}\, d\rt \,dr \right),
  \end{align}
  where $J_{\nu}$ and $H_{\nu}^{(1)}$ are as above.  Because each term
  in equation~\eqref{eq:pair-expansion} meromorphically continues to
  the Riemann surface $\Lambda$, it suffices to show that the partial
  sums of the series converge locally (in $\lambda$) uniformly (in
  $j$).

  By the asymptotic expansions of Bessel functions for large order, we
  know~\cite[10.19]{DLMF} that, locally in $\lambda \in \Lambda$, and for $r\in \supp
  \chi$,
  \begin{align*}
    J_{\nu}(\lambda r) &= \frac{1}{\sqrt{2\pi \nu}} \left(
    \frac{e \lambda r}{2\nu}\right)^{\nu} + \lo\left( \frac{1}{\sqrt{\nu}}
    \left( \frac{e \lambda r}{2\nu}\right)^{\nu}\right), \\
    H_{\nu}^{(1)} (\lambda r) &= \frac{1}{i} \sqrt{\frac{2}{\pi \nu}}
                                \left( \frac{e \lambda
                                r}{2\nu}\right)^{-\nu} +
                                \lo\left( \frac{1}{\sqrt{\nu}}
                                \left( \frac{e \lambda r}{2\nu}\right)^{-\nu}\right),
  \end{align*}
  as $\nu \to \infty$ through the positive reals.  In particular,
  for $j$ large enough, each term in
  equation~\eqref{eq:pair-expansion} can be bounded by
  \begin{align*}
    &C \int_{0}^{\infty}\int_{0}^{r} \frac{1}{\pi \nu_{j}} \chi(r)
    \chi(\rt) f_{j}(\rt) g_{j}(r) \left[ \left(
        \frac{\rt}{r}\right)^{\nu_{j}} (1 + o (1))\right] (\rt
    r)^{\frac{n}{2}}\, d\rt \, dr \\
    &\quad \quad+ C\int_{0}^{\infty}\int_{r}^{\infty} \frac{1}{\pi \nu_{j}} \chi(r)
    \chi(\rt) f_{j}(\rt) g_{j}(r) \left[ \left(
        \frac{r}{\rt}\right)^{\nu_{j}} (1 + o (1))\right] (\rt
    r)^{\frac{n}{2}}\, d\rt \, dr.
  \end{align*}
  Observe that in the first integral, $\rt / r$ is bounded by $1$,
  while $r/\rt$ is bounded by $1$ in the second.  

  Because $\chi$ is compactly supported, we may therefore bound each
  term (for $j$ large enough) by
  \begin{equation*}
    \frac{C_{\chi}}{\nu_{j}} \norm[L^{2}]{f_{j}} \norm[L^{2}]{g_{j}}.
  \end{equation*}
  This sequence is absolutely summable, so the partial sums of the
  series in equation~\eqref{eq:pair-expansion} converge locally
  uniformly.  This establishes that the cut-off resolvent on the full cone
  ($a=0$) meromorphically extends to the logarithmic cover $\Lambda$
  of the complex plane.

  We now proceed to the case of the truncated cone ($a > 0$).  We
  proceed by an appeal to the analytic Fredholm theorem.  
  
  Fix $\chi_{0}, \chi_{\infty} \in C^{\infty}((a,\infty))$ so that
  $\chi_{0}(r)$ is supported near $r=a$, $\chi_{\infty}(r)$ is
  identically zero near $r=a$, and $\chi_{0} + \chi_{\infty} = 1$.
  We let $R_{\infty}(\lambda)$ denote the resolvent on the non-truncated
  cone and $R_{0}(\lambda)$ denote the resolvent on a compact manifold
  with boundary into which the support of $\chi_{0}$ embeds
  isometrically.  We define the parametrix
  \begin{equation*}
    Q(\lambda) = \tilde{\chi}_{0} R_{0}(\lambda) \chi_{0} +
    \tilde{\chi}_{\infty} R_{\infty}(\lambda) \chi_{\infty},
  \end{equation*}
  where $\tilde{\chi}$ have similar support properties and are
  identically $1$ on the support of their counterparts.  Applying
  $\lap + \lambda^{2}$ yields a remainder of the form $I + \sum [\lap,
  \tilde{\chi}_{i}] R_{i}(\lambda) \chi_{i}$.  Both terms are compact
  and the operator is invertible for large $\Im \lambda$ by Neumann
  series, so applying $R_{a}(\lambda)$ to both sides and inverting the
  remainder shows that it has a meromorphic continuation.  
\end{proof}

\section{Proof of Theorem~\ref{thm:main-thm}}
\label{sec:reson-count-trunc}

By the formula for the resolvent in Theorem~\ref{thm:resolvent}, the
resonances of $R_{a}(\lambda)$ correspond to those $\lambda$ for which
$H_{\nu_{j}}^{(1)}(\lambda a) = 0$ for some $j$.  For simplicity we
will discuss only the case $a=1$ as the other cases can be found by
rescaling.  As mentioned in the introduction, we consider only those
resonances nearest to the upper half-plane, i.e., those with
\begin{equation}
  \label{eq:arg-restriction}
  -\frac{\pi}{2} < \arg \lambda < 0 \quad \text{or} \quad \pi <
  \arg \lambda < \frac{3\pi}{2}.
\end{equation}

Because $\nu_{j}$ is real, we may relate the zeros of
$H_{\nu_{j}}^{(1)}(\lambda)$ in the region given by
equation~\eqref{eq:arg-restriction} to zeros of
$H^{(2)}_{\nu_{j}}(\lambda)$ in the quadrant
$0 < \arg \lambda < \frac{\pi}{2}$ via analytic continuation formulae.
Indeed, it is well-known~\cite[10.11.5, 10.11.9]{DLMF} that
\begin{align}
  \label{eq:connection-formulae}
  H^{(1)}_{\nu}(z e^{\pi}) &= -e^{-\nu \pi \imath} H^{(2)}_{\nu}(z),
  \\
  H^{(1)}_{\nu}(\overline{z}) &= \overline{H^{(2)}_{\nu}(z)}. \notag
\end{align}
The first of these equations identifies zeros of $H^{(1)}_{\nu}$ in $\pi <
\arg\lambda < \frac{3\pi}{2}$ to zeros of $H^{(2)}_{\nu}$ in the
first quadrant; the second equation does the same for zeros of
$H^{(1)}_{\nu}$ with $-\frac{\pi}{2}< \arg \lambda < 0$.  In
particular, each zero of $H^{(2)}_{\nu}$ with $0 \leq \arg \lambda
\leq \pi / 2$ corresponds to exactly two resonances.

For large enough $\nu$, the zeros of the Hankel function
$H^{(2)}_{\nu}$ in the first quadrant lie near the boundary of (a
scaling of) an ``eye-like'' domain $K\subset \complexes$.  The domain
$K$ is symmetric about the real axis and is bounded by the following
curve and its conjugate:
\begin{equation*}
  z = \pm (t \coth t - t^{2})^{1/2} + i (t^{2} - t \tanh t)^{1/2},
  \quad 0 \leq t \leq t_{0},
\end{equation*}
where $t_{0}$ is the positive root of $t = \coth t$.  We refer to the
piece of the boundary of $K$ lying in the upper half-plane by $\pd
K_{+}$.  

The constant $A_{n}$ given above is given by the following:
\begin{equation}
  \label{eq:defn-ofAn}
  A_{n} = \frac{2(n-1) \vol(B_{n-1})}{n (2\pi )^{n}} \int_{\pd K_{+}} \frac{|1-z^{2}|^{1/2}}{|z|^{n+1}}\,d|z|,
\end{equation}
where $B_{n-1}$ is the $(n-1)$-dimensional unit ball.  Observe that,
up to a factor of the volume of the unit sphere (which is replaced by
the volume of $Y$ in the theorem statement), the constant $A_{n}$
is the same constant computed by Stefanov~\cite{Stefanov}.

We use below two different parametrizations of the piece of $\pd
K_{+}$ lying the in the quadrant $0 \leq \arg z \leq \pi / 2$.  The
first parametrization is by the argument of $z$, i.e., by the map
\begin{equation*}
  \left[ 0 , \frac{\pi}{2}\right] \to \pd K_{+}, \quad \theta = \arg z
  \mapsto z = z(\theta).
\end{equation*}

For the second parametrization, we introduce the function $\rho$,
defined by
\begin{equation}
  \label{eq:rho-defn}
  \rho(z) =\frac{2}{3} \zeta^{3/2} =  \log \frac{1 + \sqrt{1-z^{2}}}{z} - \sqrt{1-z^{2}}, \quad
  |\arg z | < \pi,
\end{equation}
where (following Stefanov~\cite[Section 4]{Stefanov} and
Olver~\cite[Chapter 10]{Olver-book}) the branches of the functions
above are chosen so that $\zeta$ is real when $z$ is.  Another
characterization is that the principal branches are chosen when $0 < z
< 1$ and continuity is demanded elsewhere.

The boundary $\pd K$ is the vanishing set of $\Re \rho$.  This yields
a parametrization of the part of $\pd K_{+}$ lying in $0\leq \arg z
\leq \pi / 2$:
\begin{equation*}
  \left[0, \frac{\pi}{2}\right] \to \pd K_{+}, \quad t \mapsto \rho
  ^{-1}(- i t) = z.
\end{equation*}
The transition between the two parametrizations is given by
\begin{equation*}
  \frac{dt}{d\theta} = \frac{dt}{dz} \frac{dz}{d\theta} = (i \rho' (z))
  (iz) = \sqrt{1-z^{2}}.
\end{equation*}

The function $\zeta$ defined in equation~\eqref{eq:rho-defn} is the
solution of the ODE
\begin{equation*}
  \left( \frac{d\zeta}{dz}\right) ^{2} = \frac{1-z^{2}}{\zeta z^{2}}
\end{equation*}
that is infinitely differentiable on the positive real axis (including
at $z=1$).  As is implicit in equation~\eqref{eq:rho-defn}, it can be
analytically continued to the complex plane with a branch cut along
the negative real axis.

Because the resonances correspond to zeros of $H^{(2)}_{\nu_{j}}$, we
must also consider the asymptotic distribution of the $\nu_{j}$.  
In what follows, we consider only the case when the periodic geodesics
of $(Y,h)$ have measure zero.\footnote{When $(Y,h)$ is a sphere, the
  analysis is simplified slightly.  In that case, one replaces the use
  of the Weyl formula with explicit
formulae for the eigenvalues $\mu_{j}^{2}$ and their multiplicities.}
The eigenvalues $\mu_{j}^{2}$ of $\lap_{h}$ obey Weyl's law:
\begin{align*}
  N_{h}(\mu) &= \# \{ \mu_{j} : \mu_{j} \leq \mu \text{ with
               multiplicity }\} \\
  &= \frac{\vol{B_{n-1}}}{(2\pi)^{n-1}}\vol (Y,h) \mu^{n-1} + R(\mu).
\end{align*}
Here $\vol (B_{n-1})$ denotes the volume of the unit ball in
$\reals^{n-1}$ and $\vol (Y,h)$ is the volume of $Y$ equipped with the
metric $h$.  In general, $R(\mu) = \bo(\mu^{n-2})$, but if we now impose
the dynamical hypothesis (that the set of periodic geodesics of
$(Y,h)$ has Liouville measure zero), then a theorem of
Duistermaat--Guillemin~\cite{DG} (in the boundaryless case) and
Ivrii~\cite{Ivrii1,Ivrii2} (in the boundary case) shows that
\begin{equation*}
  R(\lambda) = \lo (\mu^{n-2}).
\end{equation*}
The non-periodicity assumption then allows us to count eigenvalues on
intervals of length one:
\begin{align*}
  N_{h}(\mu, \mu+1) &=\#\{ \mu_{j} : \mu \leq \mu_{j} \leq \mu + 1
                      \text{ with multiplicity }\} \\
  &= (n-1) \frac{\vol (B_{n-1})}{(2\pi)^{n-1}}\vol(Y,h)\mu^{n-2} + \lo (\mu^{n-2}).
\end{align*}
As $\nu_{j}^{2} = \mu_{j}^{2} + (n-2)^{2} / 4$, the same counting
formula holds for $\nu_{j}$, i.e.,
\begin{align}
  \label{eq:nu-counting}
  N_{\nu} (\rho, \rho + 1) &= \# \{ \nu_{j} : \rho \leq \nu_{j} \leq
  \rho + 1 \text{ with multiplicity }\} \notag\\
  &= (n-1) \frac{\vol (B_{n-1})}{(2\pi)^{n-1}}\vol(Y,h)\rho^{n-2} + \lo (\rho^{n-2}).
\end{align}

We now turn our attention to the zeros of the Hankel function
$H_{\nu}^{(2)}(z)$ with $\arg z \in [0 , \pi / 2]$.  An argument from
Watson~\cite[pages 511--513]{Watson} is easily adapted to give a
precise count of the number of zeros of $H^{(2)}_{\nu}$ in this
sector.  Indeed, that argument shows that the number of zeros is given
by the closest integer to $\nu / 2 - 1/4$ (when $\nu - 1/2$ is an
integer, there is a zero on the imaginary axis and so rounds up).

As $\nu \to \infty$ through positive real values, we have an
asymptotic expansion~\cite[10.20.6]{DLMF} relating the Hankel function
to the Airy function
\begin{equation}
  \label{eq:h2-airy}
  H^{(2)}_{\nu}(\nu z) \sim 2e^{i\pi / 3} \left(
    \frac{4\zeta}{1-z^{2}}\right)^{1/4} \left( \frac{\Ai (e^{-2\pi
        i/3}\nu^{2/3}\zeta)}{\nu^{1/3}} \sum_{k=0}^{\infty}
    \frac{A_{k}(\zeta)}{\nu^{2k}} + \frac{\Ai '(e^{-2\pi
        i/3}\nu^{2/3}\zeta)}{\nu^{5/3}} \sum_{k=0}^{\infty}\frac{B_{k}(\zeta)}{\nu^{2k}}\right).
\end{equation}
Here $A_{k}$ and $B_{k}$ are real and infinitely differentiable for
$\zeta \in \reals$.  This expansion is uniform in $|\arg z| \leq \pi - \delta$ for fixed
$\delta > 0$.  In particular, for large enough $\nu$, the zeros of the
Hankel function are well-approximated by zeros of the Airy function
and we may identify each zero $h_{\nu, k}$ of the Hankel function
$H^{(2)}_{\nu}$ with a zero of the Airy function $\Ai(-z)$.  

Let $a_{k}$ denote the $k$-th zero of the Airy function $\Ai (-z)$; all
$a_{k}$ are positive and
\begin{equation*}
  a_{k} = \left[ \frac{3}{2} \left(k \pi -
    \frac{\pi}{4}\right)\right]^{2/3} + \bo (k^{-4/3}).
\end{equation*}

We now define $\lambda_{\nu, k}$  and $\widetilde{\lambda}_{\nu,k}$
via the Airy zeros and their leading approximations:
\begin{align*}
  \lambda _{\nu, k} &= \nu \zeta ^{-1}(\nu^{-2/3} e^{-i
                      \frac{\pi}{3}}a_{k}) = \nu \rho^{-1}\left(
                      -i\frac{2}{3} a_{k}^{3/2} \nu^{-1}\right) \\
  \widetilde{\lambda}_{\nu,k} &= \nu \rho^{-1} \left( - i \left( k -
                                \frac{1}{4}\right)\pi \nu^{-1}\right), 
\end{align*}
where $k =  1, \dots, \floor{\nu /2 + 1/4}$.
By the Hankel expansion~\eqref{eq:h2-airy}, $|h_{\nu, k} -
\lambda_{\nu, k}| \leq C / \nu$ for large enough $\nu$ while $|h_{\nu,
k}  - \widetilde{\lambda}_{\nu, k}|  \leq C / \nu$ for large enough
$\nu$ and $k$.  As we have identified $\floor{\nu / 2 + 1 /4}$
approximate zeros, we can conclude that these account for all $h_{\nu,
k}$.  

We now divide our attention into those zeros with small argument and
those with large argument.  We introduce the auxiliary counting function
\begin{equation*}
  N(r, \theta_{1}, \theta_{2}) = \# \{ \sigma : \sigma \text{ is a
    resonance with } |\sigma | \leq r, \arg \sigma \in [ \theta_{1}, \theta_{2}]\}.
\end{equation*}

We first address those with small
argument.  Fix $\epsilon > 0$ and consider those zeros with $|z| < r$
and $\arg z \in [0,\epsilon]$.  We need count those $\lambda_{\nu,k}$
with $\arg \lambda _{\nu, k} \in [ 0, \epsilon]$ and $|\lambda _{\nu,
  k} | \leq r$.  As $|\lambda_{\nu,k}|$ is comparable to $\nu$, we can
overcount these zeros by counting all $\lambda_{\nu, k}$ with argument
in $[0, \epsilon]$ and $\nu \leq Cr$.

Because $|\rho| \leq C \epsilon^{3/2}$ for those $\lambda_{\nu, k}$
with $\arg \lambda _{\nu, k} \in [0,\epsilon]$, we must only count
those $a_{k}$ with $a_{k} \leq C \nu^{2/3}\epsilon$.  The leading
order asymptotic~\cite[9.9.6]{DLMF} for the zeros of the Airy function
shows that this number is $\bo (\nu \epsilon^{3/2})$.

We now count those resonances with argument in $[0,\epsilon]$.
Putting together the asymptotic for $\nu_{j}$ in
equation~\eqref{eq:nu-counting} with the previous two paragraphs, we
have (with $m(\nu_{j})$ denoting the multiplicity of $\nu_{j}$)
\begin{align}
  \label{eq:small-arg-est}
  N(r, 0, \epsilon) &= \sum_{j = 1}^{\infty}m(\nu_{j}) \#\left\{ h_{\nu_{j}, k} :
                      |h_{\nu_{j},k}| \leq r , \arg h_{\nu_{j},k} \in
                      [ 0, \epsilon]\right\}  \notag\\
                    &\leq \sum_{j=1}^{Cr} m(\nu_{j}) C\nu_{j} \epsilon^{3/2} \notag\\
                    &\leq C\epsilon^{3/2} \sum_{\rho = 0}^{Cr}\sum_{\nu_{j} \in [\rho,
                      \rho+1]}m(\nu_{j}) \rho \leq C \epsilon^{3/2}r^{n}. 
\end{align}

We now consider those resonances with argument in $[\epsilon , \pi /
2]$.  For large enough $\nu$, the approximations
$\widetilde{\lambda}_{\nu,k}$ are valid for these resonances.  We
count those approximate resonances with $\nu_{j} \in [ \rho, \rho +
1)$ and $\arg \lambda_{\nu,k} \in [\theta, \theta + \Delta\theta]$.  We
start by introducing, for fixed $\nu$, the number $\Delta k_{\nu}$ of
$\widetilde{\lambda}_{\nu, k}$ with argument lying in $[\theta, \theta
+ \Delta \theta]$.  Observe that the definition of
$\widetilde{\lambda}_{\nu,k}$ relates $\Delta k_{\nu}$ with $\Delta t$
by
\begin{equation*}
  \Delta k_{\nu} = \frac{\nu}{\pi}\Delta t + \bo (1),
 \end{equation*}
where $\Delta t$ denotes the change in $t$ corresponding to $\Delta
\theta$ in the parametrizations above.  Note that $\Delta t$ is
\emph{independent} of the choice of $\nu$.  We can then write
\begin{align*}
  \# \left\{ \widetilde{\lambda}_{\nu,k} : \nu_{j} \in [\rho, \rho+1), \arg
  \widetilde{\lambda}_{\nu,k} \in [\theta , \theta + \Delta \theta]\right\} & =
                                                                 \sum_{\rho
                                                                 \leq
                                                                 \nu_{j}
                                                                 \leq
                                                                  \rho+1}
                                                                 m(\nu_{j})
                                                                 \Delta
                                                                 k_{\nu}
  \\
                                                               &=
                                                                 \sum_{\rho
                                                                 \leq
                                                                 \nu_{j}
                                                                 <
                                                                 \rho
                                                                 + 1}
                                                                 m(\nu_{j})
                                                                 \left(
                                                                 \frac{\nu_{j}}{\pi}
                                                                 \Delta
                                                                 t +
                                                                 \bo (1)\right) 
\end{align*}

By the definition of the approximate zeros
$\widetilde{\lambda}_{\nu,k}$, we can estimate their size
$|\widetilde{\lambda}_{\nu,k}|$ in terms of $|z(\theta)|$, provided
that $\arg \widetilde{\lambda}_{\nu,k}\in [\theta, \theta + \Delta
\theta]$, yielding
\begin{equation*}
  |\widetilde{\lambda}_{\nu,k}|  = \nu \left( |z(\theta)| + \bo ( \Delta \theta)\right).
\end{equation*}
In particular, if $\nu_{j}|z(\theta)| \geq r$ but $|\lambda _{\nu, k}
| \leq r$, then $\nu_{j} \in \left[ \frac{r}{|z(\theta)|}(1 - c\Delta
  \theta), \frac{r}{|z(\theta)|}\right]$.  We may thus rewrite our
counting function as follows:
\begin{align*}
  \#\left\{
  \widetilde{\lambda}_{\nu,k} :
  |\widetilde{\lambda}_{\nu,k}|
  \leq r , \arg
  \widetilde{\lambda}_{\nu, k}
  \in [\theta, \theta
  + \Delta \theta]\right\}
  &=
    \sum_{\substack{|\widetilde{\lambda}_{\nu, k}| \leq r \\ \arg \widetilde{\lambda}_{\nu, k} \in
  [\theta, \theta + \Delta \theta]}} m(\nu_{j}) \\
  &= \sum_{\substack{\nu_{j}|z(\theta)| \leq r \\ \arg \widetilde{\lambda}_{j,k}
  \in [\theta, \theta + \Delta \theta]}} m(\nu_{j}) +
  \sum_{\substack{\nu_{j} \in \left[ \frac{r}{|z(\theta)|} (1 - c
  \Delta \theta) , \frac{r}{|z(\theta)|}\right] \\ \arg \widetilde{\lambda}_{\nu,
  k} \in [\theta, \theta + \Delta \theta]}} m(\nu_{j}).
\end{align*}
By our improved Weyl's law~\eqref{eq:nu-counting}, the second term is
$\bo (r^{n-2})$.

We now focus our attention on the first term (here $\floor{\cdot}$
denotes the ``floor'' function):
\begin{align*}
  \sum_{\substack{\nu_{j}|z(\theta)| \leq r \\ \arg \widetilde{\lambda}_{j,k}
  \in [\theta, \theta + \Delta \theta]}} m(\nu_{j}) &= \sum_{\rho =
                                                      0}^{\floor{r/|z|
                                                      -
                                                      1}}\sum_{\nu_{j}\in[\rho,
                                                      \rho + 1)}
                                                      \sum_{\arg
                                                      \widetilde{\lambda}_{\nu,k}\in
                                                      [\theta, \theta
                                                      + \Delta
                                                      \theta]}
                                                      m(\nu_{j}) +
                                                      \sum_{\nu_{j}
                                                      \in
                                                      [\floor{r/z},
                                                      r/z]}
                                                      \sum_{\arg\widetilde{\lambda}_{\nu,k}\in
                                                      [\theta, \theta
                                                      + \Delta
                                                      \theta]}
                                                      m(\nu_{j}) \\
  &= \sum_{\rho = 0}^{\floor{r/|z| - 1}} \sum_{\nu_{j} \in [\rho,
    \rho+1)}m(\nu_{j}) \Delta k_{\nu} +          \sum_{\nu_{j}
                                                      \in
                                                      [\floor{r/z},
                                                      r/z]}
                                                      \sum_{\arg\widetilde{\lambda}_{\nu,k}\in
                                                      [\theta, \theta
                                                      + \Delta
                                                      \theta]}
    m(\nu_{j})  .
\end{align*}
Again by Weyl's law, we observe that the second term is $\bo
(r^{n-2})$.  By relating $\Delta t$ and $\Delta k_{\nu}$ we can
rewrite the first term:
\begin{align*}
  \sum_{\rho = 0}^{\floor{r/|z| - 1}} \sum_{\nu_{j} \in [\rho,
    \rho+1)}m(\nu_{j}) \Delta k_{\nu}  &= \sum_{\rho =
                                         0}^{\floor{r/|z|-1}}
                                         \sum_{\nu_{j} \in [\rho, \rho
                                         + 1)}m(\nu_{j})
                                         \frac{\nu_{j}}{\pi} \Delta t
                                         + \sum_{\nu_{j} \leq
                                         \floor{r/|z|}} m(\nu_{j}) O(1).
\end{align*}
By Weyl's law~\eqref{eq:nu-counting}, the second term is $\bo (r^{n-1})$, so
we again consider the first term.

As $\Delta t$ is independent of $\nu_{j}$, we may use
Weyl's law as well on the first term:
\begin{align*}
  \sum_{\rho =0}^{\floor{r/|z|-1}} \sum_{\nu_{j} \in [\rho, \rho  +
  1)}m(\nu_{j}) \frac{\nu_{j}}{\pi} \Delta t &= \sum_{\rho =
                                               0}^{\floor{r/|z|-1}}\left[\frac{n-1}{2^{n-1}\pi^{n}}\vol(B_{n-1})\vol(Y,h)\rho^{n-1}
                                               \Delta t + \bo (\rho
                                               ^{n-2}) +
                                               \lo(\rho^{n-1}) \Delta
                                               t\right] \\
  &= \frac{2(n-1)}{(2\pi)^{n}}\vol(B_{n-1}) \vol (Y,h) \Delta t \sum_{\rho
    = 0}^{\floor{r/|z|-1}} \rho^{n-1} + \bo (r^{n-1}) +
    \lo(r^{n})\Delta t \\
  &= \frac{2(n-1)}{(2\pi)^{n}n}\vol(B_{n-1})\vol(Y,h) \frac{1}{n}\left(
    \frac{r}{|z(\theta)|}\right)^{n} \Delta t + \bo (r^{n-1}) +
    \lo(r^{n}) \Delta t .
\end{align*}

We finally introduce a Riemann sum in $t$ to understand this main term:
\begin{align}
  \label{eq:main-est}
  \# \{ \widetilde{\lambda}_{\nu,k} &: |\widetilde{\lambda}_{\nu,k}|\leq r , \arg
                                      \widetilde{\lambda}_{\nu,k} \in [\epsilon, \pi/2]\} \\
                                    &=
                                                       \int_{t^{-1}(\epsilon)}^{\pi
                                                       / 2} \left(
                                                       \frac{2(n-1)\vol(B_{n-1})}{(2\pi)^{n}
                                                       n} \vol (Y, h) \right)
                                                       \frac{r^{n}}{|z(\theta)|^{n}}
                                                       \, dt + \bo
                                                       (r^{n-1}) + \lo
                                                       (r^{n}) \notag \\
  &= \frac{(n-1) \vol(B_{n-1})}{(2\pi)^{n}n }\vol(Y,h) r^{n} \int_{\pd
    K_{+}} \frac{1}{|z(\theta)|^{n}} \,dt + \bo (\epsilon r^{n})  +
    \lo (r^{n}) \notag \\
  &= \left(\frac{(n-1)\vol(B_{n-1})}{(2\pi)^{n} n} \vol (Y,h) \int_{\pd
    K_{+}} \frac{|1-z^{2}|^{1/2}}{|z|^{n+1}} \, d|z|\right) r^{n} +
    \bo (\epsilon r^{n}) + \lo (r^{n})\notag \\
  &=  A_{n} \vol(Y,h) r^{n} + \bo (\epsilon r^{n}). + \lo(r^{n}) \notag 
\end{align}
Here the prefactor of $2$ disappeared because the first integral
parametrizes only half of $\pd K_{+}$.  It reappears in the statement
of Theorem~\ref{thm:main-thm} because each zero here corresponds to
two resonances (one on each sheet).  We further observe that the
constant $A_{n}\vol(Y,h)$ agrees with the leading term found in the
Euclidean case found by Stefanov~\cite{Stefanov}.

Sending $\epsilon$ to $0$ establishes the theorem for the approximate
zeros $\lambda_{\nu,k}$.  Because each $\lambda_{\nu,k}$ is in a
$C/\nu$ neighborhood of a zero $h_{\nu,k}$, this finishes the proof of
the theorem.

\section*{Acknowledgments}
\label{sec:acknowledgements}

Part of this research formed the core of the second author's Master's
project at Texas A\&M University.  DB acknowledges partial support
from NSF grants DMS-1500646 and DMS-1654056.  The authors also
thank David Borthwick, Tanya Christiansen, Colin Guillarmou, and Jeremy Marzuola for helpful conversations.

% \bibliographystyle{plain}
% \bibliography{rescone}

\begin{thebibliography}{10}

\bibitem{borthwick:2010}
David Borthwick.
\newblock Sharp upper bounds on resonances for perturbations of hyperbolic
  space.
\newblock {\em Asymptot. Anal.}, 69(1-2):45--85, 2010.

\bibitem{borthwick:2012}
David Borthwick.
\newblock Sharp geometric upper bounds on resonances for surfaces with
  hyperbolic ends.
\newblock {\em Anal. PDE}, 5(3):513--552, 2012.

\bibitem{Borthwick-Philipp}
David Borthwick and Pascal Philipp.
\newblock Resonance asymptotics for asymptotically hyperbolic manifolds with
  warped-product ends.
\newblock {\em Asymptot. Anal.}, 90(3-4):281--323, 2014.

\bibitem{CT1}
Jeff Cheeger and Michael Taylor.
\newblock On the diffraction of waves by conical singularities. {I}.
\newblock {\em Comm. Pure Appl. Math.}, 35(3):275--331, 1982.

\bibitem{CT2}
Jeff Cheeger and Michael Taylor.
\newblock On the diffraction of waves by conical singularities. {II}.
\newblock {\em Comm. Pure Appl. Math.}, 35(4):487--529, 1982.

\bibitem{DLMF}
{\it NIST Digital Library of Mathematical Functions}.
\newblock http://dlmf.nist.gov/, Release 1.0.21 of 2018-12-15.
\newblock F.~W.~J. Olver, A.~B. {Olde Daalhuis}, D.~W. Lozier, B.~I. Schneider,
  R.~F. Boisvert, C.~W. Clark, B.~R. Miller and B.~V. Saunders, eds.

\bibitem{DG}
J.~J. Duistermaat and V.~W. Guillemin.
\newblock The spectrum of positive elliptic operators and periodic
  bicharacteristics.
\newblock {\em Invent. Math.}, 29(1):39--79, 1975.

\bibitem{Ivrii1}
V.~Ja. Ivri\u{\i}.
\newblock The second term of the spectral asymptotics for a
  {L}aplace-{B}eltrami operator on manifolds with boundary.
\newblock {\em Funktsional. Anal. i Prilozhen.}, 14(2):25--34, 1980.

\bibitem{Ivrii2}
V.~Ya. Ivri\u{\i}.
\newblock Exact spectral asymptotics for elliptic operators acting in vector
  bundles.
\newblock {\em Funktsional. Anal. i Prilozhen.}, 16(2):30--38, 96, 1982.

\bibitem{Olver-book}
F.~W.~J. Olver.
\newblock {\em Asymptotics and special functions}.
\newblock Academic Press [A subsidiary of Harcourt Brace Jovanovich,
  Publishers], New York-London, 1974.
\newblock Computer Science and Applied Mathematics.

\bibitem{Stefanov}
Plamen Stefanov.
\newblock Sharp upper bounds on the number of the scattering poles.
\newblock {\em J. Funct. Anal.}, 231(1):111--142, 2006.

\bibitem{Watson}
G.~N. Watson.
\newblock {\em A {T}reatise on the {T}heory of {B}essel {F}unctions}.
\newblock Cambridge University Press, Cambridge, England; The Macmillan
  Company, New York, 1944.

\end{thebibliography}

\end{document}